\documentclass[11pt]{article}
\usepackage{amsmath, amssymb, amsthm}
\usepackage{verbatim}
\usepackage{multicol}
\usepackage{hyperref}
\usepackage{enumerate}
\usepackage{tikz}
\usetikzlibrary{positioning,matrix}
\usepackage{comment}

\date{}
\title{\vspace{-0.8cm}On the Tur\'an number of ordered forests}
\author{
D\'aniel Kor\'andi \thanks{Institute of Mathematics, EPFL, Lausanne, Switzerland. Research supported in part by SNSF grants 200020-162884 and 200021-175977. Emails: \{daniel.korandi,istvan.tomon\}@epfl.ch.}
\and
G\'abor Tardos \thanks{R\'enyi Institute, Budapest, Hungary. Research supported in part by the ``Lend\"ulet'' project of the Hungarian Academy of Sciences and by the National Research, Development and Innovation Office, NKFIH, projects K-116769 and SNN-117879. Part of this research was done while visiting EPFL. Email: tardos@renyi.hu.}
\and
Istv\'an Tomon \footnotemark[1]
\and
Craig Weidert \thanks{Email: craig.weidert@gmail.com}
}

\theoremstyle{plain}
\newtheorem{theorem}{Theorem}[section]

\newtheorem{corollary}[theorem]{Corollary}

\newtheorem{lemma}[theorem]{Lemma}
\newtheorem{conjecture}[theorem]{Conjecture}

\theoremstyle{remark}
\newtheorem*{remark}{Remark}

\DeclareMathOperator{\ex}{ex}

\newcommand{\eps}{\varepsilon}
\newcommand{\subs}{\subseteq}

\begin{document}

\maketitle
\sloppy

\begin{abstract}
An ordered graph $H$ is a simple graph with a linear order on its vertex set. 
The corresponding Tur\'an problem, first studied by Pach and Tardos, asks for the maximum number $\ex_<(n,H)$ of edges in an ordered graph on $n$ vertices that does not contain $H$ as an ordered subgraph.
It is known that $\ex_<(n,H) > n^{1+\eps}$ for some positive $\eps=\eps(H)$
unless $H$ is a forest that has a proper 2-coloring with one color class
totally preceding the other one.
Making progress towards a conjecture of Pach and Tardos, we prove that $\ex_<(n,H) =n^{1+o(1)}$ holds for all such forests that are ``degenerate'' in a certain sense. This class includes every forest for which an $n^{1+o(1)}$ upper bound was previously known, as well as new examples. Our proof is based on a density-increment argument.
\end{abstract}

\section{Introduction}
An \emph{ordered graph} $G$ is defined as a triple $(V,E,<)$ where $(V,E)$ is a simple graph and $<$ is a linear order on the vertex set $V$. We say that an ordered graph $H=(V',E',<')$ is an \emph{ordered subgraph} of $G$ if there is an order preserving embedding of $V'$ into $V$ that maps edges to edges. If $G$ does not contain $H$ as an ordered subgraph then we say that $G$ is \emph{$H$-free}.

The following Tur\'{a}n-type question arises naturally: For a fixed ordered graph $H$, what is the maximum number of edges that an $H$-free ordered graph on $n$ vertices can have? This maximum, called the \emph{extremal} or \emph{Tur\'an number} of $H$, is denoted by $\ex_{<}(n,H)$. The systematic study of extremal numbers was initiated by Pach and Tardos in \cite{PT06}. 

\medskip
For usual (unordered) graphs, the Erd\H{o}s-Stone-Simonovits theorem says that the extremal number of a graph is controlled by its chromatic number.
As it turns out, ordered graphs exhibit a similar phenomenon. The \emph{interval chromatic number} $\chi_<(H)$ of an ordered graph $H=(V,E,<)$ is the smallest integer $r$, such that $V$ can be split into $r$ intervals (i.e., sets of consecutive vertices in the ordering) such that no edge of $H$ has both endpoints in the same interval. It is not hard to show (see \cite{PT06})
that $\ex_<(n,H)= \left(1-\frac{1}{\chi_<(H)-1}\right)\binom{n}{2} + o(n^2)$. This asymptotically determines the extremal number of $H$ when $\chi_<(H)\ge 3$. However, much like for usual graphs, the problem becomes more difficult when $\chi_<(H)=2$.

Our work focuses on the general problem of classifying ordered graphs that have close to linear extremal numbers, i.e., satisfy $\ex_<(n,H)=n^{1+o(1)}$. Note that if $H$ contains some cycle of length $k$ then $\ex_<(n,H)\ge \Omega(n^{1+1/k})$. Indeed, it is well-known (see, e.g., \cite{BBOOK}) that there are $n$-vertex (unordered) graphs with $\Omega(n^{1+1/k})$ edges that do not contain any $k$-cycle.\footnote{For odd $k$ we even have $\ex_<(n,H)=\Omega(n^2)$.} So if $\ex_<(n,H)=n^{1+o(1)}$ holds then $H$ is acyclic with interval chromatic number 2.  Pach and Tardos conjectured that the converse holds, as well. More precisely, they made the following stronger conjecture:

\begin{conjecture}[Pach--Tardos \cite{PT06}] \label{conj:main}
Let $H$ be an acyclic ordered graph such that $\chi_<(H)=2$. Then
\[ \ex_<(n,H) = n(\log n)^{O(1)}.\]
\end{conjecture}  

In this paper we make some progress towards proving this conjecture by showing that $\ex_<(n,H)=n^{1+o(1)}$ holds for a large class of ordered forests $H$.
\medskip

It will be more convenient to state our result in the language of pattern-avoiding matrices. Let us now describe the analogous problem in this context. A 0-1 matrix $A$ is a matrix with all entries from $\{0,1\}$. Its \emph{weight}
$w(A)$ is the number of 1-entries in it. We say that $A$ \emph{contains} another 0-1 matrix $B$ if either $B$ is a submatrix of $A$ or it can be obtained from a submatrix of $A$ by changing some 1-entries to 0-entries. In other words, $A$ contains $B$ if we can get $B$ by deleting some rows, columns and 1-entries from $A$. 
We denote this relation by $B\prec A$. A pattern in our context is just a fixed 0-1 matrix $P$, and the corresponding Tur\'an-type problem is then to maximize the weight of an $n$-by-$n$ 0-1 matrix that does not contain $P$. Let $\ex(n,P)$ denote this maximum weight. Note that this extremal function is not defined for patterns with zero weight.

We can think of a pattern $P$ as the bipartite adjacency matrix of some
ordered graph $H_P$ of interval chromatic number 2, where the order of the
vertices is inherited from the order of the corresponding rows and columns of
$P$, and row vertices precede column vertices. Then $\ex(n,P)$
translates to the maximum number of edges in an $H_P$-free ordered graph $G$
on $2n$ vertices, such that all edges of $G_A$ connect
the first $n$ vertices to the last $n$ vertices.\footnote{For equality to hold, we actually need
the mild extra assumption that there is a 1-entry in the last row of $P$
and also in the first column of $P$. Otherwise $H_P=H_Q$ for some pattern
$Q\ne P$, so avoiding $H_P$ in the ordered graph means avoiding both $P$
and $Q$ in the 0-1 matrix.} If not for this last condition, this would be the exact same quantity as $\ex(2n,H_P)$. And indeed, as it was observed in \cite{PT06}, the two functions are closely related:
\begin{equation} \label{eq:equiv}
  \ex(\lfloor n/2\rfloor,P) - O(n) \le \ex_{<}(n,H_P) = O(\ex(n,P)\log n).
\end{equation}

Let us call a pattern $P$ \emph{acyclic} if $H_P$ does not contain any cycle. It is easy to see that $P$ is acyclic if and only if it contains no submatrix $P'$ such that every row and column of $P'$ contains at least two 1-entries. Equation \eqref{eq:equiv} shows that Conjecture~\ref{conj:main} can be stated in the following, equivalent form:
 
\begin{conjecture}[Pach--Tardos \cite{PT06}] \label{conj:mainmx}
Let $P$ be an acyclic pattern. Then $\ex(n,P)=n(\log n)^{O(1)}$.
\end{conjecture}  

A stronger variant of this conjecture had earlier been proposed by F\"{u}redi and Hajnal \cite{FH92}, who thought $\ex(n,P)=O(n\log n)$ holds for every acyclic pattern $P$. However, this conjecture was refuted by Pettie \cite{P11}, who constructed an acyclic pattern $P_{0}$ such that $\ex(n,P_{0})=\Omega(n\log n\log \log n)$.

\medskip
There are several patterns that are known to satisfy Conjecture~\ref{conj:mainmx}. For example, it is shown in \cite{PT06} that if $P$ is obtained from a pattern $P'$ by appending a new last column with a single 1-entry, then 
\begin{equation} \label{eq:col}
\ex(n,P)=O(\ex(n,P')\log n)
\end{equation}
holds, except for the trivial cases when $w(P)=0$ (and thus $\ex(n,P)$ is not defined) or $P=(1)$ is the 1-by-1 matrix of weight 1 (when $\ex(n,P)=0$). Using this and similar operations, the conjecture has been verified for a large family of matrices \cite{FH92,PT06,T05}. These include all patterns of weight up to 6, with essentially two exceptions (omitting 0-entries for clarity):
\begin{center}
\begin{tikzpicture}
\matrix[matrix of math nodes, left delimiter={(}, right delimiter={)},label=left:{$Q_1=$~~~}] (A) at (-3,0) {
  1 &   & 1 &   \\
  1 &   &   & 1 \\
    & 1 &   & 1 \\
};
\draw[thick,opacity=.3] (A-1-3.center) -- (A-1-1.center) -- (A-2-1.center) -- (A-2-4.center) -- (A-3-4.center) -- (A-3-2.center); 

\node at (-.5,0) {and};

\matrix[matrix of math nodes, left delimiter={(}, right delimiter={)},label=left:{$Q_2=$~~~}] (B) at (3,0) {
    & 1 &   & 1 \\
  1 &   & 1 &   \\
  1 &   &   & 1 \\
};
\draw[thick,opacity=.3] (B-1-2.center) -- (B-1-4.center) -- (B-3-4.center) -- (B-3-1.center) -- (B-2-1.center) -- (B-2-3.center); 
\end{tikzpicture}
\end{center}
In certain special cases, e.g., when $P$ is a permutation matrix \cite{MT04} or a double permutation matrix (obtained from a permutation matrix by doubling each column) \cite{G09}, even the stronger bound $\ex(n,P)=O(n)$ is known to hold.

\medskip
Let us now define a broad class of patterns that includes, up to transposing, all matrices that are known to satisfy Conjecture~\ref{conj:mainmx}. We say that an $l$-by-$k$ 0-1 matrix $P$ is \emph{vertically separable} if it can be cut along a horizontal line without separating the 1-entries in more than one column. In other words, if there exists $1\le a<l$ such that for all but at most one column $1\le y\le k$, we have $P(x,y)=0$ either for every $x\le a$ or for every $x>a$. In this case, we say that $P$ is \emph{separated} into the upper part induced by the first $a$ rows and the lower part induced by the last $l-a$ rows.

We call a matrix $P$ \emph{vertically degenerate} if it can be partitioned into its rows using vertical separations. Equivalently, $P$ is vertically degenerate if every submatrix $P'$ of $P$ either has a single row or is a vertically separable matrix. Note that vertically degenerate matrices are always acyclic.

The reader might find it helpful to visualize a pattern $P$ as a graph whose vertices are the 1-entries of $P$, and two 1-entries are connected if they are in the same row or column with only 0-entries between them (see $Q_1$ and $Q_2$ above). A pattern is acyclic if and only if this graph is acyclic. Vertically separable means that the matrix can be split into two parts along a horizontal line that cuts through at most one (vertical) edge. We consider this edge to be destroyed by the cut and this may make the way for subsequent cuts. A pattern is vertically degenerate if it can be split into its rows by applying a series of such cuts.

Our main result is the following theorem that says that every vertically degenerate pattern has close to linear extremal number. Some preliminary results from this work establishing $\ex(n,Q_1)\le n^{1+o(1)}$ have previously appeared in the Master's thesis of the fourth author \cite{W09}.

\begin{theorem} \label{thm:main}
Let $P$ be a vertically degenerate matrix. Then $\ex(n,P)\le n^{1+o(1)}$. 
\end{theorem}
 
This theorem can be thought of as a common generalization of all previously known results about acyclic patterns, albeit with a somewhat weaker upper bound. It also applies to many new matrices, including all $3\times k$ acyclic patterns and among them $Q_1$ and $Q_2$. For $Q_2$, no bound better than $O(n^{5/3})$ was previously known. (Note that the bound $\ex(n,P)=O(n^{5/3})$ follows from the K\H{o}v\'ari-S\'os-Tur\'an theorem~\cite{KST} for all $3$-by-$k$ patterns $P$, even for non-acyclic ones.) As discussed above, this also implies that every pattern $P$ of weight up to 6 satisfies $\ex(n,P)\le n^{1+o(1)}$.

We will make the $o(1)$ term in Theorem~\ref{thm:main} explicit by showing $\ex(n,P)=n2^{O(\log^cn)}$ for a constant $c\le1-1/l$ (see Theorem~\ref{thm:explicit}). The proof is essentially a density-increment argument that starts with an $n\times n$ matrix $A$ with large weight and shows that either $A$ contains $P$ or it contains a significantly denser submatrix. 
Section~\ref{sec:embed} contains the heart of the inductive proof. Here we restrict our attention to certain well-structured matrices. Our general result is then reduced to this special case in Section~\ref{sec:clean} by finding the necessary well-structured submatrix inside any 0-1 matrix with large enough weight. We finish the paper with some remarks in Section~\ref{sec:remarks}.
\medskip

\noindent{\bf Notation.} Throughout this paper, $\log$ stands for the binary logarithm. As usual, $[n]$ denotes the set $\{1,\dots,n\}$ and $[m,n]=\{m,m+1,\dots,n\}$. We write $\{0,1\}^{m\times n}$ for the set of 0-1 matrices with $m$ rows and $n$ columns. For $A\in\{0,1\}^{m\times n}$ and $I\subs [m]$, $J\subs [n]$ we write $A(I\times J)$ to denote the submatrix of $A$ induced by the rows in $I$ and columns in $J$.

\section{The special case of $(k,u)$-complete 0-1 matrices} \label{sec:embed}

Let $k,m,n,u\in\mathbb N$ and let $A$ be an $m$-by-$kn$ 0-1 matrix. We consider $A$ to be the union of $k$ \emph{vertical blocks} $A([m]\times[(j-1)n+1,jn])$ for $j\in[k]$. We say that $A$ is \emph{$(k,u)$-complete}, if among the $n$ entries in the intersection of any row and any vertical block, one always finds at least $u$ 1-entries.

Let $A$ be an $m$-by-$kn$ $(k,u)$-complete matrix and let
$Q\in\{0,1\}^{l\times k}$ be an $l$-by-$k$ pattern.  If $Q\prec A$ in such a
way that each column of $Q$ ``comes from'' a different vertical block of $A$,
we say $Q$ has a \emph{block-respecting embedding} in $A$. More precisely, a block-respecting embedding of $Q$ is a pair of functions $(f,g)$ such that $f:[l]\to[m]$ is strictly increasing, $g:[k]\to[kn]$ satisfies $(j-1)n<g(j)\le jn$ for all $j\in[k]$ and (in order to make this an embedding) all the 1-entries of $Q$ map to a 1-entry in $A$, that is $A(f(i),g(j))=1$ whenever $Q(i,j)=1$.

The key property of this notion is that block-respecting embeddings of the upper and lower parts of a vertically separated pattern can be easily combined into such an embedding of the whole pattern, as shown by the following lemma.

\begin{lemma}\label{lem:combine}
Let $P\in\{0,1\}^{l\times k}$ be a pattern and let $a\in[l-1]$. Suppose $P'=P([a]\times[k])$ has a block-respecting embedding $(f',g')$ into a $(k,u)$-complete matrix $A\in\{0,1\}^{m\times kn}$, and $P''=P([a+1,l]\times[k])$ has a block-respecting embedding $(f'',g'')$ in $A$. If $f'(a)<f''(1)$, and $g'(b)=g''(b)$ whenever both $P'$ and $P''$ have a 1-entry in some column $b\in[k]$, then $P$ also has a block-respecting embedding in $A$.
\end{lemma}

\begin{proof}
The two embeddings can be combined into a single block-respecting embedding $(f,g)$ of $P$ in $A$ as follows:
\[f(i)=\left\{\begin{array}{ll}
f'(i)&\textrm{if $i\le a$}\\ f''(i-a)&\textrm{otherwise}
\end{array}\right.\;\;\;\;\;g(j)=\left\{\begin{array}{ll}
g'(j)&\textrm{if $P(i,j)=1$ for some $i\in[a]$}\\g''(j)&\textrm{otherwise.}
\end{array}\right.\]
\end{proof}

We will find block-respecting embeddings of vertically degenerate patterns $P$ in $(k,u)$-complete matrices inductively, starting with single rows and gradually combining them using the above lemma. For this, we will need that $A$ does not have a submatrix that is too dense. Let us start with classifying vertically degenerate patterns.

We call a 0-1 matrix $P$ a \emph{class-$0$} matrix if it has a single row. For $s>0$ we call a 0-1 matrix \emph{class-$s$} if it has a single row or it can be partitioned into two class-$(s-1)$ patterns by a vertical separation. Clearly, all class-$s$ patterns are vertically degenerate and all $l$-by-$k$ vertically degenerate patterns are class-$(l-1)$.

Let $h:\mathbb N\to\mathbb R$ be a function. We call a 0-1 matrix $A$ \emph{$h$-sparse} if for every positive integer $n$, all $n$-by-$n$ submatrices of $A$ have weight at most $n\cdot h(n)$. We will use this definition for the function $h(n)=h_{b,c,d}(n)=d2^{b\log^cn}$ for some positive constants $b$, $d$ and $0<c<1$. 

The following inequality is easy to verify for every $0<m<n$ using elementary calculus:
\begin{equation}\label{bounding}
\frac{h_{b,c,d}(n)}{h_{b,c,d}(m)}>\left(\frac nm\right)^{bc\log^{c-1}n}
\end{equation}

\begin{lemma}\label{main}
Let $h=h_{b,c,d}$ for some fixed positive constants $b,d$ and $0<c<1$. Let $k,m,n,s$ and $u$ be positive integers satisfying $m\le n$ and $u\le h(n)$ such that $x=bc\log^{c-1}n\le 1/10$, and suppose $A\in\{0,1\}^{m\times kn}$ is a $(k,u)$-complete $h$-sparse matrix. If
\[ m\ge40n^{1-x^s}\left(\frac{h(n)}u\right)^2 \]
then every class-$s$ pattern with $k$ columns has a block-respecting embedding in $A$.
\end{lemma}

\begin{remark} The bound on $m$ in the lemma can be reformulated as
\[ \frac nm\left(\frac{h(n)}u\right)^2\le\frac{n^{x^s}}{40}. \]
Both the ratios $n/m$ and $h(n)/u$ are assumed to be at least 1. This
inequality bounds them from above. If the right-hand side dips below 1, then
the condition is not satisfiable. Therefore, in the proof below, we assume
$n^{x^s}\ge40$.
\end{remark}
\begin{proof}
We proceed by induction on $s$. A pattern $P$ consisting of a single row clearly has
a block-respecting embedding in every row of $A$. We may therefore assume that $s>0$,
and fix a class-$s$ pattern $P\in\{0,1\}^{l\times k}$ that is vertically separable at $a\in[l-1]$ into two class-$(s-1)$ patterns: the upper part $P'=P([a]\times[k])$ and the lower part $P''=P([a+1,l]\times[k])$. Let us choose a $b\in[k]$ so that no column $b'$ with $b\ne b'\in[k]$ has a 1-entry in both parts.

Set $m^*=\left\lceil 3^3\cdot 40n^{1-x^{s-1}}(h(n)/u)^2\right\rceil$ and $\beta=\lfloor m/m^*\rfloor$. 
Here $m\ge m^*$ and hence $\beta\ge 1$, because $n^{x^{s-1}-x^s}=n^{(1/x-1)x^s}\ge n^{9x^s}\ge 40^9\ge 3^3$ (using $x<1/10$ and $n^{x^s}\ge 40$). Recall that $A$ is partitioned into $k$ vertical blocks. We also partition most of $A$ into $\beta$ \emph{horizontal} blocks $A_i=A([(i-1)m^*+1,im^*]\times[kn])$ for $i\in[\beta]$. Note that each horizontal block is a $(k,u)$-complete $h$-sparse matrix with enough rows for the induction hypothesis to ensure by that both $P'$ and $P''$ have block-respecting embeddings in it.

Let $S_i'=\{g(b):(f,g)\textrm{ is a block-respecting embedding of $P'$ in }A_i\}$ be the set of column indices from the $b$'th vertical block that appear in a block-respecting embedding of $P'$ in $A_i$. We define the matrix $A_i'$ contained in $A_i$ as follows. As a first step, we turn every 1-entry of the columns in $S_i'$ into 0. In the second step, we delete all rows that have fewer than $u/3$ 1-entries in the $b$'th vertical block. The first step ensures that $P'$ has no block-respecting embedding in $A_i'$, while the second step makes $A_i'$ a $(k,\lceil u/3\rceil)$-complete matrix. As $A_i'$ is contained in $A$, it is also $h$-sparse. The inductive hypothesis then implies that $A_i'$ has fewer than $m^*/3$ rows, so more than $2m^*/3$ rows have been removed.

We define $S_i''$ and $A_i''$ analogously to $S_i'$ and $A_i'$, but using the block-respecting embeddings of $P''$. We conclude that the number of rows in $A_i''$ is also less than $m^*/3$.

Now let $S_i=S_i'\cap S_i''$. If the sets $S_i$ are \emph{not} pairwise disjoint, then we have a block-respecting embedding of $P$ in $A$, as required. Indeed, suppose that for some $i'<i''$ and $j$ we have $j\in S_{i'}\cap S_{i''}$. Then $j\in S_{i'}'$, and hence, by definition, we have a block-respecting embedding $(f',g')$ of $P'$ in $A_{i'}$ with $g'(b)=j$. Similarly, $j\in S_{i''}''$ means that we have a block-respecting embedding $(f'',g'')$ of $P''$ in $A_{i''}$ with $g''(b)=j$. These embeddings (more precisely, the corresponding embeddings in $A$) satisfy the conditions of Lemma~\ref{lem:combine}, and therefore can be combined into a block-respecting embedding of $P$ in $A$.

Thus, we may assume that the sets $S_i$ are pairwise disjoint. Note that these sets are all contained in the interval of length $n$ corresponding to the $b$'th vertical block, so their average size is at most $n/\beta$. Fix an $i$ with $|S_i|\le n/\beta$ and let $B=A_i([m^*]\times S_i)$. If a row was removed from $A_i$ in the second step of the construction of both $A_i'$ and $A_i''$, then this row contains more than $u/3$ 1-entries in $B$. Indeed, any row of the $(k,u)$-complete matrix $A_i$ contains at least $u$ 1-entries in the $b$'th vertical block, but if a row is removed during the construction of $A'_i$, then fewer than $u/3$ of these 1-entries are outside $S_i'$, and similarly, fewer than $u/3$ of them are outside $S''_i$ if the row was also removed during the construction of $A''_i$.

As we have seen above, at least $2m^*/3$ rows of $A_i$ were removed in the second step of constructing each of $A'_{i}$ and $A''_{i}$, so more than $m^*/3$ rows were removed for both of them. Hence $B$ contains more than $m^*/3$ rows with more than $u/3$ 1-entries, implying $w(B)>um^*/9$. If the number of columns $|S_i|$ of $B$ is less than $m^*$, then let $B^*$ be an $m^*$-by-$m^*$ submatrix of $A$ that contains $B$. Clearly, $w(B^*)\ge w(B)$. Otherwise, let $B^*$ be the $m^*$-by-$m^*$ submatrix of $B$ of maximum weight. By averaging, we have $w(B^*)\ge (m^*/|S_i|)w(B)\ge(\beta m^*/n)w(B)$ in this case. Either way, the weight of our $B^*$ is at least $(\beta m^*/n)w(B)\ge u{m^*}^2\beta/(9n)$. We will show that this value is more than $m^*h(m^*)$. This will then contradict our assumption that the matrix $A$ is $h$-sparse (as $B^*$ is a submatrix of $A$) and finish the proof of the lemma.

The inequality $m^*h(m^*)<u{m^*}^2\beta/(9n)$, or equivalently, $h(m^*)/u<m^*\beta/(9n)$ is now easy to check. In fact, all the bounds in the statement of the lemma were chosen to facilitate this calculation. First of all, \eqref{bounding} implies $h(m^*)<h(n)(m^*/n)^x$ and hence 
\[ \frac{h(m^*)}{u}<\frac{h(n)}u\left(\frac{m^*}n\right)^x. \]
Then using the definition of $m^*$, the assumption $x\le1/10$, the lower bound on $m$, and finally $m^*\beta\ge m/2$, we get
\begin{align*}
\frac{h(m^*)}{u} &\le\frac{h(n)}u\left(1200n^{-x^{s-1}}\left(\frac{h(n)}u\right)^2\right)^x = 1200^xn^{-x^s}\left(\frac{h(n)}u\right)^{2x+1}\\
&<2.1n^{-x^s}\left(\frac{h(n)}u\right)^2<\frac m{18n}\le\frac{m^*\beta}{9n}
\end{align*}
as needed.\end{proof}

\section{Reduction to the special case} \label{sec:clean}

In this section we prove an explicit version of Theorem~\ref{thm:main}:

\begin{theorem}\label{thm:explicit}
For any class-$s$ pattern $P$, we have
\[ \ex(n,P)\le n2^{O(\log^{\frac s{s+1}}n)}. \]
\end{theorem}

\begin{proof} For class-0 (single-row) patterns $P$, the theorem claims $\ex(n,P)=O(n)$. This clearly holds for all such patterns, because any $n$-by-$n$ matrix of weight at least $kn$ contains a row with $k$ 1-entries.
Let $s\ge 1$ and fix a class-$s$ pattern $P\in\{0,1\}^{l\times k}$ with $l,k\ge2$. We will prove the theorem by showing that $\ex(n,P)\le nh(n)$ holds for all $n$ with $h(n)=h_{b,c,d}(n)=d2^{b\log^cn}$, where $c=s/(s+1)$, and $b=b(c,P)$ and $d=d(b,c,P)$ are constants to be chosen later as follows.

Our proof proceeds by analyzing a (hypothetical) minimal counterexample, and finds a contradiction if $b$ is large enough in terms of $c$ and $P$. However, this argument only works if this minimal counterexample has size at least $n_0=n_0(b,c,P)$ that only depends on $b,c$ and $P$. By choosing $d$ large enough, we can make sure that $h(n)\ge n$ and hence $\ex(n,P)\le nh(n)$ holds for every $n\le n_0$, guaranteeing that a minimal counterexample has size at least $n_0$.

\smallskip
So suppose for contradiction that $\ex(n,P)\le nh(n)$ does not always hold, and let $n$ be the smallest integer violating this inequality. Then there is a matrix $A\in\{0,1\}^{n\times n}$ such that $P$ is not contained in $A$, the weight of $A$ is $w(A)>nh(n)$, but any proper submatrix of $A$ is $h$-sparse.

Let $x=bc\log^{c-1}n$ be the exponent in the estimate~\eqref{bounding} and set $n^*=\left\lfloor n/(6k)^{1/x}\right\rfloor$. Then \eqref{bounding} implies $h(n^*)<\frac{h(n)}{6k}$, and we also have $x\le 1$ and hence $1\le n^*<n/k$ if $n$ is large enough. Set $\alpha=\lfloor n/n^*\rfloor\ge k$ and let $B$ be the $n$-by-$\alpha n^*$ submatrix of $A$ formed by the $\alpha n^*$ columns of largest weight in $A$. Then $w(B)\ge\frac{\alpha n^*}{n}w(A)>\frac{nh(n)}{2}$. We partition the matrix $B$ into $\alpha$ \emph{vertical blocks}, each consisting of $n^*$ consecutive columns. We call the intersection of a vertical block and a row a \emph{box}, so the matrix $B$ consists of $\alpha n$ boxes. The \emph{weight} of a box is the number of 1-entries in it. We distinguish three classes of boxes according to their weight:
We say that a box is \emph{light} if its weight is below $h(n^*)/\alpha$,
\emph{heavy} if its weight is over $h(n)/(6k)$,
and \emph{regular} if it is neither heavy, nor light.
\smallskip

The total weight of all light boxes is at most $\alpha n(h(n^*)/\alpha)=nh(n^*)<nh(n)/(6k)$.

If some vertical block contained $n^*$ heavy boxes, then we could form an $n^*$-by-$n^*$ submatrix of weight over $n^*h(n)/(6k)>n^*h(n^*)$, contradicting the assumption that each proper submatrix of $A$ is $h$-sparse. Therefore, we can form $n^*$-by-$n^*$ submatrices, one in each vertical block, that together contain all the heavy boxes. By the same sparsity condition, each of these matrices have weight at most $n^*h(n^*)$. This makes the total weight of the heavy boxes at most $\alpha n^*h(n^*)<nh(n)/(6k)$.

The total weight of regular boxes is $w(B)$ minus the total combined weight of heavy and light boxes. By the above calculations, this is at least $nh(n)/6$, therefore the number of regular boxes is at least $\frac{nh(n)}{6}/\frac{h(n)}{6k}=kn$. Let $r_i$ be the number of regular boxes in row $i$, so we have $\sum_{i=1}^nr_i\ge kn$. 
Let us say that a row is good for a $k$-set $X$ of vertical blocks if it has a regular box in each block in $X$. Note that a row $i$ is good for $\binom{r_i}k\ge r_i-(k-1)$ $k$-sets, so on average
\[ \frac{\sum_{i=1}^n\binom{r_i}k}{\binom\alpha k}\ge\frac{\sum_{i=1}^nr_i-(k-1)n}{\binom\alpha k}\ge\frac n{\alpha^k}. \]
rows are good for a $k$-set of vertical blocks. We fix a $k$-set $X$ of blocks
such that at least $m^*=\lceil n/\alpha^k \rceil$ rows are good for it. Let $T$ be a set of $m^*$ such rows and let $S$ be the set of columns in the vertical blocks of $X$. Then $C=B(T,S)$ is an $m^*$-by-$kn^*$ $(k,u^*)$-complete submatrix for $u^*=\lceil h(n^*)/\alpha\rceil$.

\smallskip

We finish the proof of the theorem by showing that the matrix $C$ satisfies the conditions of Lemma~\ref{main}, yet it violates its statement.

We have already observed that $C$ is $(k,u^*)$-complete and (as a proper submatrix of $A$) it is also $h$-sparse. On the other hand, $P$ does not have a block-respecting (or any) embedding in $C$, as otherwise $A$ would also contain $P$, which we assumed not to be the case. Note that we also have $u^*=\lceil h(n^*)/\alpha\rceil \le h(n^*)$ and $m^*=\lceil n/\alpha^k \rceil\le \lceil n/(2\alpha)\rceil\le n^*$. So to get a contradiction from Lemma~\ref{main}, it is enough to show that for an appropriately chosen $b$,
\[ m^*\ge 40{n^*}^{1-{x^*}^s}\left(\frac{h(n^*)}{u^*}\right)^2, \]
with $x^*=bc\log^{c-1}n^*<1/10$. 

Using $n^*<n$, $x^*>x$, $u^*\ge h(n^*)/\alpha$ and $m^*\ge n/\alpha^k$, we obtain
\[
\frac{40{n^*}^{1-{x^*}^s}\left(\frac{h(n^*)}{u^*}\right)^2}{m^*} \le \frac{40n^{1-x^s}\left(\frac{h(n^*)}{h(n^*)/\alpha}\right)^2}{n/\alpha^k} = 40\alpha^{k+2}n^{-x^s}.
\]
To prove that the right-hand side does not exceed 1, we will show that its logarithm is negative. In the calculation we use $\alpha\le n/n^*<2(10k)^{1/x}$ and $x=bc\log^{c-1}n=bc/\log^{1/(s+1)}n$:
\begin{align*}
\log(40\alpha^{k+2}n^{-x^s})&<6+(k+2)\log\alpha-x^s\log n\\
&<6+(k+2)\log(20k)/x-x^s\log n\\
&=6+\left(\frac{(k+2)\log(20k)}{bc}-b^sc^s\right)\log^{\frac1{s+1}}n.
\end{align*}
Then indeed, choosing $b=(k+2)\log(20k)/c$ makes this last expression negative for any parameters $k,s\ge 1$ and for all $n>1$.

Finally, the condition $x^*<1/10$ is equivalent to $n^*>2^{(10bc)^{s+1}}$, which clearly holds if $n$ (and hence $n^*$) is large enough.

This brings us to a contradiction whenever $n\ge n_0$, where the threshold $n_0$ only depends on $b,c$ and $P$. But as we mentioned above, we can choose the parameter $d$ so that a minimal counterexample would definitely satisfy $n\ge n_0$.
\end{proof}

\section{Concluding remarks} \label{sec:remarks}

Note that both of the acyclic weight-6 patterns $Q_1$ and $Q_2$ are class-2 patterns but not class-1 patterns. Theorem~\ref{thm:explicit} gives an $n2^{O(\log^{2/3}n)}$ bound for their extremal functions. As mentioned in the introduction, no non-trivial bound was known for $\ex(n,Q_2)$, but as a preliminary result of this present work, the bound $\ex(n,Q_1)=n2^{O(\sqrt{\log n\log\log n})}$ was proved in the Master's thesis of the fourth author \cite{W09}. The slightly better bound of
\begin{equation} \label{eq:q1}
\ex(n,Q_1)=n2^{O(\sqrt{\log n})}
\end{equation}
can also be proved using the techniques of Section~\ref{sec:clean} if instead of Lemma~\ref{main}, we use the following simple and elementary fact: \emph{If $m>2n/u$, then $Q_1$ has a block-respecting embedding in every $(4,u)$-complete matrix $A\in\{0,1\}^{m\times4n}$.}

In fact, we believe that our methods can be used to prove a slightly stronger variant of Theorem~\ref{thm:explicit} where the same bounds are claimed but the pattern classes are defined in the following, more relaxed, way. Once again, class-0 are those with a single row, but for $s\ge 1$, we call all patterns class-$s$ that are partitioned into class-$(s-1)$ patterns when \emph{all} the possible vertical separations (i.e., the horizontal cut that each destroy at most one vertical edge) are applied simultaneously. With this definition, $Q_1$ is class-1 because it can be partitioned into its three rows in one step, so \eqref{eq:q1} is implied by this stronger variant of our theorem. However $Q_2$ is still not class-1. To keep this paper simple, we decided not to work out the details of this argument.
\medskip

It is clear that transposing a matrix does not change its extremal number: $\ex(n,P)=\ex(n,P^T)$. In particular, Theorem~\ref{thm:main} holds for the analogously defined \emph{horizontally degenerate} patterns, as well. The smallest acyclic pattern that is neither horizontally, nor vertically degenerate (and hence not covered by our theorem) is the following 4-by-4 ``pretzel''-like matrix (again omitting 0-entries for clarity):
\begin{center}
\begin{tikzpicture}
\matrix[matrix of math nodes, left delimiter={(}, right delimiter={)},label=left:{$R=$~~~}] (A) at (0,0) {
  1 &   &   & 1 \\
    & 1 &   &   \\
  1 &   & 1 &   \\
    & 1 &   & 1 \\
};
\draw[thick,opacity=.3] (A-2-2.center) -- (A-4-2.center) -- (A-4-4.center) -- (A-1-4.center) -- (A-1-1.center) -- (A-3-1.center) -- (A-3-3.center); 
\end{tikzpicture}
\end{center}
It would be very interesting to obtain nontrivial estimates on $\ex(n,R)$.

$R$ is neither horizontally nor vertically separable. However, our approach already breaks down for matrices that are degenerate in a weaker sense, namely that every submatrix is separable, but whether it is horizontally or vertically separable depends on the submatrix. This class of patterns, including the following spiral-like pattern $S$, might be easier to handle than $R$, and estimating their extremal numbers would certainly be a natural next step towards Conjecture~\ref{conj:mainmx}.

\begin{center}
\begin{tikzpicture}
\matrix[matrix of math nodes, left delimiter={(}, right delimiter={)},label=left:{$S=$~~~}] (A) at (0,0) {
  1 &   & 1 &   & 1 \\
    & 1 &   & 1 &  \\
  1 &   &   &   &   \\
    & 1 &   &   & 1 \\
};
\draw[thick,opacity=.3] (A-2-4.center) -- (A-2-2.center) -- (A-4-2.center) -- (A-4-5.center) -- (A-1-5.center) -- (A-1-3.center) -- (A-1-1.center) -- (A-3-1.center); 
\end{tikzpicture}
\end{center}
\medskip

In terms of ordered graphs of interval chromatic number 2, vertically separable means that the the vertices of the first interval $A$ can be split into two subintervals $A_1$ and $A_2$ such that there is at most one vertex in the second interval $B$ that has neighbors in both $A_1$ and $A_2$. An ordered graph $G$ is then vertically degenerate if every ordered subgraph of $G$ is vertically separable.

Pach and Tardos \cite{PT06} conjectured that every unordered graph $G_0$ has an ordering such that $\ex_<(n,G)\le O(\ex(n,G_0)\log n)$, where $\ex(n,G_0)$ stands for the usual, unordered Tur\'an number of $G_0$. In similar spirit, note that in any acyclic ordered graph of interval chromatic number 2, we can rearrange the vertices of the first interval to get a vertically degenerate graph. Or in terms of matrices, we can make any acyclic pattern vertically degenerate by permuting its rows. This implies the following:

\begin{corollary}
Let $P$ be an acyclic pattern. We can rearrange the rows of $P$ such that the resulting pattern $P'$ satisfies $\ex(n,P')\le n^{1+o(1)}$.
\end{corollary}

Note that if we were allowed to permute the columns, as well, then we could easily get $\ex(n,P')\le n(\log n)^{O(1)}$. Indeed, we would then be able make the last column (or row) end up with a single 1-entry, and then apply induction using \eqref{eq:col}.

\end{document}